\documentclass{amsart}
\usepackage{amssymb}
\usepackage{amsfonts}
\usepackage{amsmath}
\usepackage{mathrsfs}
\usepackage[numbers,sort&compress]{natbib}

\newtheorem{theorem}{Theorem}[section]
\newtheorem{lemma}[theorem]{Lemma}
\theoremstyle{definition}
\newtheorem{definition}[theorem]{Definition}

\newtheorem{proposition}[theorem]{Proposition}
\newtheorem{notation}[theorem]{Notation}

\newtheorem{remark}[theorem]{Remark}
\numberwithin{equation}{section}

\input{tcilatex}

\begin{document}
\title[Eigenvalues of the Diffusion Operator]{Multiplicities of Eigenvalues
of the Diffusion Operator with Random Jumps from the Boundary}
\author{Jun Yan}
\address{School of Mathematics, Tianjin University, Tianjin, 300354, P.
R. China}
\email{junyantju@126.com}
\author{Guoliang Shi}
\address{School of Mathematics, Tianjin University, Tianjin, 300354, P.
R. China}
\email{glshi@tju.edu.cn}
\date{\today }
\thanks{ *This research was supported by the National Natural Science
Foundation of China under Grant No. 11601372.}

\begin{abstract}
This paper deals with a non-self-adjoint differential operator which is
associated with a diffusion process with random jumps from the boundary. Our
main result is that the algebraic multiplicity of an eigenvalue is equal to
its order as a zero of the characteristic function $\Delta(\lambda) $. This can be used to determine the
multiplicities of eigenvalues for concrete operators.
\end{abstract}

\maketitle

\section{Introduction}

This article investigates the non-self-adjoint differential operator $L$ in $%
L_{w}^{2}(J,%
\mathbb{C}
)$ generated by the differential expression
\begin{equation*}
Ly=ly:=b_{0}(x)y^{\prime \prime }+b_{1}(x)y^{\prime }
\end{equation*}%
and%
\begin{equation*}
\text{dom}(L):=\left\{ y\in L_{w}^{2}(J,%
\mathbb{C}
)\left\vert
\begin{array}{l}
y,y^{\prime }\in AC[0,1],Ly\in L_{w}^{2}(J,%
\mathbb{C}
) \\
y(0)=\int_{0}^{1}y(x)\mathrm{d}\nu _{0}(x),\text{ }y(1)=\int_{0}^{1}y(x)%
\mathrm{d}\nu _{1}(x)%
\end{array}%
\right. \right\} .
\end{equation*}%
Here $\nu _{0},\nu _{1}$ are probability distributions on $J:=(0,1)\ $and%
\begin{equation*}
w:=-\frac{1}{b_{0}},\text{ }\frac{b_{1}}{b_{0}}\in L^{1}(J,%
\mathbb{C}
),\text{ }b_{0}<0\text{ a.e. on }(0,1).
\end{equation*}%
It is well known that the operator $L$ is associated with a diffusion
process with jumping boundary which can be easily described. In this
process, whenever the boundary of the interval $\left[ 0,1\right] $ is
reached, the diffusion gets redistributed in $(0,1)$ according to the
probability distributions $\nu _{0},\nu _{1},$ runs again until it hits the
boundary, gets redistributed and repeats this behavior forever. Due to its
probabilistic significance, the process leads to several interesting results
(see, e.g., \cite{iddo,iddo2,ilie,ilie2,kolb3,wen,feller,diffusion,kolb2}
and the references therein).

Let us mention that in the case of $b_{0}(x)\equiv -1,$ $b_{1}(x)\equiv 0,$ $%
\nu _{0}=\nu _{1}=\delta _{a},$ $a\in (0,1),$ M. Kolb and D. Krej\v{c}i\v{r}%
\'{\i}k in \cite{diffusion} analyzed the geometric and algebraic
multiplicities of the eigenvalues from a purely operator-theoretic
perspective and showed that all the eigenvalues of $L$ are algebraically
simple if, and only if, $a\notin
\mathbb{Q}
$. Based on this, they studied the basis properties of $L$. This is our
starting point and we aim to develop a further result on the multiplicities
of eigenvalues of $L$ in a general setting.

Let $y_{1}(x,\lambda )$ and $y_{2}(x,\lambda )\ $be the fundamental
solutions of
\begin{equation}
b_{0}(x)y^{\prime \prime }(x)+b_{1}(x)y^{\prime }(x)=\lambda y(x)
\label{fangcheng}
\end{equation}%
determined by the initial conditions
\begin{equation*}
y_{1}(0,\lambda )=y_{2}^{\prime }(0,\lambda )=1,y_{1}^{\prime }(0,\lambda
)=y_{2}(0,\lambda )=0.
\end{equation*}%
Denote%
\begin{equation}
\Delta (\lambda ):=\det \left(
\begin{array}{cc}
\int_{0}^{1}y_{1}(x,\lambda )\mathrm{d}\nu _{0}(x)-1 & \int_{0}^{1}y_{2}(x,%
\lambda )\mathrm{d}\nu _{0}(x) \\
\int_{0}^{1}y_{1}(x,\lambda )\mathrm{d}\nu _{1}(x)-y_{1}(1,\lambda ) &
\int_{0}^{1}y_{2}(x,\lambda )\mathrm{d}\nu _{1}(x)-y_{2}(1,\lambda )%
\end{array}%
\right) .  \label{chara}
\end{equation}%
Then direct calculation yields that $\lambda$ is an eigenvalue of $L$ if and
only if $\Delta (\lambda )=0.$

Let us now present the main theorem of this paper.

\begin{theorem}
\label{equal}Assume $\lambda _{0}$ be an eigenvalue of $L$ with algebraic
multiplicity $\chi (\lambda _{0}).\ $Let $n_{0}$ denote the order of $%
\lambda _{0}$ as a zero of $\Delta (\lambda )$. Then $\chi (\lambda
_{0})=n_{0}.$
\end{theorem}

We would like to emphasize that this theorem is useful for identifying the\
multiplicities of eigenvalues of the operator $L.$ For example, it provides
a straightforward method to obtain the main result in \cite[Theorem 1]%
{diffusion} (see Remark \ref{extend diffusion}). Moreover, suppose $b_{0}$
and $b_{1}\ $be constants and $\nu _{0}=\nu _{1}=\delta _{\frac{1}{2}},$
then as a consequence of Theorem \ref{equal}, Remark \ref{open diffusion}
shows us that all the eigenvalues of $L$ are algebraically simple. This
partially answers an open problem listed in \cite{diffusion}.

\section{Basic Properties and Preliminaries}

Let us first recall some notations and definitions.

\begin{notation}
Let $T$ be a linear operator in a Hilbert space $H$. In what follows, dom$%
(T) $, $\ker (T)$ are the domain, the kernel of $T$, respectively$;$ $\sigma
(T), $ $\sigma _{p}(T),$ $\rho (T),$ denote the spectrum, point spectrum,
the resolvent set of $T$, respectively$;$ $R_{\lambda }(T)$ := $(T-\lambda
I)^{-1}$, $\lambda \in \rho (T)$, is the resolvent of $T$.
\end{notation}

\begin{definition}
Let $T$ be a linear operator in a Hilbert space $H$. The smallest integer $%
p>0$ such that $\ker (T^{p})=$ $\ker (T^{p+1})$ is called the ascent of $T$
and it is denoted by $\alpha (T)$.
\end{definition}

\begin{definition}
Let $T$ be a closed linear operator in a Hilbert space $H$ and let $\lambda
_{0}$ be an eigenvalue of $T.$ Then the space $\ker ((T-\lambda _{0}I))$ is
called the eigenspace of $T$ corresponding to $\lambda _{0},$ and its
dimension is called the \textbf{geometric multiplicity} of $\lambda _{0}$.
The space $\underset{n=1}{\overset{\infty }{\cup }}\ker ((T-\lambda
_{0}I)^{n})$ is called the generalized eigenspace of $T$ corresponding to $%
\lambda _{0},$ with its dimension referred to as the \textbf{algebraic
multiplicity} of $\lambda _{0}.$
\end{definition}

In this section, we mainly prove the following proposition which will be
used in the proof of our main theorem.

\begin{proposition}
\label{bounded}The operator $L$ is closed and has a purely discrete
spectrum. Moreover, for any point $\lambda _{0}\in \sigma ( L) ,$ $\alpha
(L-\lambda _{0}I)$ is finite.
\end{proposition}

In order to prove Proposition \ref{bounded}, we first consider the
differential operator $L_{0}$ in $L_{w}^{2}(J,%
\mathbb{C}
)$ defined by
\begin{eqnarray*}
L_{0}y:&=&b_{0}(x)y^{\prime \prime }+b_{1}(x)y^{\prime }, \\
\text{dom}(L_{0}):&=&\left\{ y\in L_{w}^{2}(J,%
\mathbb{C}
)\left\vert
\begin{array}{l}
y,y^{\prime }\in AC[0,1],L_{0}y\in L_{w}^{2}(J,%
\mathbb{C}
), \\
y(0)=y(1)=0%
\end{array}%
\right. \right\} .
\end{eqnarray*}%
It is well known that
\begin{equation*}
(R_{\lambda }(L_{0})f)(x)=\int_{0}^{1}G_{\lambda }^{0}(x,t)f(t)\mathrm{d}t,%
\text{ }x\in \left[ 0,1\right] ,\text{ }f\in L_{w}^{2}(J,%
\mathbb{C}
)\text{ }
\end{equation*}%
where%
\begin{equation*}
G_{\lambda }^{0}(x,t)=\left\{
\begin{array}{l}
\frac{y_{2}(t,\lambda )\left[ y_{2}(x,\lambda )y_{1}(1,\lambda
)-y_{1}(x,\lambda )y_{2}(1,\lambda )\right] }{b_{0}(t)W(t)y_{2}(1,\lambda )}%
,0\leq t\leq x, \\
\frac{y_{2}(x,\lambda )\left[ y_{2}(t,\lambda )y_{1}(1,\lambda
)-y_{1}(t,\lambda )y_{2}(1,\lambda )\right] }{b_{0}(t)W(t)y_{2}(1,\lambda )}%
,x\leq t\leq 1%
\end{array}%
\right\} .
\end{equation*}%
Here $W(x)=\exp (-\int_{0}^{x}\frac{b_{1}(t)}{b_{0}(t)}\mathrm{d}t)\ $is the
Wronskian of $y_{1}$ and $y_{2}.$ It is obvious that $R_{\lambda }(L_{0}):$ $%
L_{w}^{2}(J,%
\mathbb{C}
)\longrightarrow $ dom$(L_{0})$ is a compact operator for $\lambda \in \rho
(L_{0})=%
\mathbb{C}
\backslash \sigma (L_{0}),$ where $\sigma (L_{0})=\left\{ \lambda
_{n}\right\} $ and $\lambda _{n}$ are zeros of the entire function $%
y_{2}(1,\lambda )$(\cite[Sec. III., Example 6.11]{kato}).

Next, we give the formula on the resolvent $R_{\lambda }(L)$, following
which Proposition \ref{bounded} can be proved directly.

\begin{lemma}
\label{resolvent}For every $\lambda \in
\mathbb{C}
\backslash \left[ \sigma (L_{0})\cup \sigma _{p}(L)\right] ,$ the resolvent $%
R_{\lambda }(L)$ of $L$ admits the following decomposition%
\begin{eqnarray}
(R_{\lambda }(L)f)(x) &=&(R_{\lambda
}(L_{0})f)(x)+g_{0}(x)\int_{0}^{1}(R_{\lambda }(L_{0})f)(x)\mathrm{d}\nu
_{0}(x)  \label{resolvent decom} \\
&&+g_{1}(x)\int_{0}^{1}(R_{\lambda }(L_{0})f)(x)\mathrm{d}\nu _{1}(x)  \notag
\end{eqnarray}%
for each $f\in L_{w}^{2}(J,%
\mathbb{C}
)\ $and $x\in \left[ 0,1\right] ,$ where
\begin{equation*}
g_{0}(x)=\frac{\left[ y_{2}(1,\lambda )-\int_{0}^{1}y_{2}(x,\lambda )\mathrm{%
d}\nu _{1}(x)\right] y_{1}(x,\lambda )-\left[ y_{1}(1,\lambda
)-\int_{0}^{1}y_{1}(x,\lambda )\mathrm{d}\nu _{1}(x)\right] y_{2}(x,\lambda )%
}{\Delta (\lambda )}
\end{equation*}%
and%
\begin{equation*}
g_{1}(x)=\frac{\left[ 1-\int_{0}^{1}y_{1}(x,\lambda )\mathrm{d}\nu _{0}(x)%
\right] y_{2}(x,\lambda )+y_{1}(x,\lambda )\int_{0}^{1}y_{2}(x,\lambda )%
\mathrm{d}\nu _{0}(x)}{\Delta (\lambda )}.
\end{equation*}
\end{lemma}

\begin{proof}
Firstly, it is easy to see that $R_{\lambda }( L) $ is a bounded operator on
$L_{w}^{2}( J,%
\mathbb{C}
) .$ In fact, the last two terms of the decomposition represent finite rank
perturbations of the compact operator $R_{\lambda }( L_{0}) . $ More
specifically, for $i=0,1,$
\begin{equation*}
g_{i}(x)\int_{0}^{1}(R_{\lambda }(L_{0})f)(x)\mathrm{d}\nu
_{i}(x)=g_{i}(x)\int_{0}^{1}\int_{0}^{1}G_{\lambda }^{0}(x,t)f(t)\mathrm{d}t%
\mathrm{d}\nu _{i}(x)
\end{equation*}%
are continuous on $\left[ 0,1\right] $ for $\lambda \in
\mathbb{C}
\backslash \left[ \sigma (L_{0})\cup \sigma _{p}(L)\right] .$

Next, we prove that $R_{\lambda }(L)f\in $ dom$(L).$ Indeed, the fact
\begin{equation*}
(R_{\lambda }(L_{0})f)(0)=(R_{\lambda }(L_{0})f)(1)=0
\end{equation*}%
yields that
\begin{eqnarray*}
(R_{\lambda }(L)f)(0) &=&\int_{0}^{1}(R_{\lambda }(L)f)(x)\mathrm{d}\nu
_{0}(x) \\
&=&\frac{y_{2}(1,\lambda )-\int_{0}^{1}y_{2}(x,\lambda )\mathrm{d}\nu _{1}(x)%
}{\Delta (\lambda )}\int_{0}^{1}(R_{\lambda }(L_{0})f)(x)\mathrm{d}\nu
_{0}(x) \\
&&+\frac{\int_{0}^{1}y_{2}(x,\lambda )\mathrm{d}\nu _{0}(x)}{\Delta (\lambda
)}\int_{0}^{1}(R_{\lambda }(L_{0})f)(x)\mathrm{d}\nu _{1}(x)
\end{eqnarray*}%
and%
\begin{eqnarray*}
&&(R_{\lambda }(L)f)(1) \\
&=&\int_{0}^{1}(R_{\lambda }(L)f)(x)\mathrm{d}\nu _{1}(x) \\
&=&\frac{y_{2}(1,\lambda )\int_{0}^{1}y_{1}(x,\lambda )\mathrm{d}\nu
_{1}(x)-y_{1}(1,\lambda )\int_{0}^{1}y_{2}(x,\lambda )\mathrm{d}\nu _{1}(x)}{%
\Delta (\lambda )}\int_{0}^{1}(R_{\lambda }(L_{0})f)(x)\mathrm{d}\nu _{0}(x)
\\
&&+\frac{\left[ 1-\int_{0}^{1}y_{1}(x,\lambda )\mathrm{d}\nu _{0}(x)\right]
y_{2}(1,\lambda )+y_{1}(1,\lambda )\int_{0}^{1}y_{2}(x,\lambda )\mathrm{d}%
\nu _{0}(x)}{\Delta (\lambda )}\int_{0}^{1}(R_{\lambda }(L_{0})f)(x)\mathrm{d%
}\nu _{1}(x).
\end{eqnarray*}%
Moreover, it is easy to deduce that
\begin{equation*}
b_{0}(x)(R_{\lambda }(L)f)^{\prime \prime }+b_{1}(x)(R_{\lambda
}(L)f)^{\prime }-\lambda (R_{\lambda }(L)f)=f\in L_{w}^{2}(J,%
\mathbb{C}
).
\end{equation*}%
Therefore, $R_{\lambda }(L)$ is a bounded operator from $L_{w}^{2}(J,%
\mathbb{C}
)\ $to dom$(L)$ and $R_{\lambda }(L)$ is the right inverse of $L-\lambda .$
It remains to show that $R_{\lambda }(L)$ is the left inverse of $L-\lambda .
$ In fact, for every $\Psi \in $ dom$(L),$ denote
\begin{equation*}
\Psi _{0}(x)=\left[ R_{\lambda }(L_{0})(L-\lambda )\Psi \right] (x),\text{ }%
\widetilde{\Psi }:=\Psi _{0}-\Psi .
\end{equation*}%
Then%
\begin{equation*}
\left\{
\begin{array}{l}
b_{0}(x)\widetilde{\Psi }^{\prime \prime }+b_{1}(x)\widetilde{\Psi }^{\prime
}=\lambda \widetilde{\Psi }, \\
\widetilde{\Psi }(0)=\int_{0}^{1}\widetilde{\Psi }(x)\mathrm{d}\nu
_{0}(x)-\int_{0}^{1}\Psi _{0}(x)\mathrm{d}\nu _{0}(x),\text{ } \\
\widetilde{\Psi }(1)=\int_{0}^{1}\widetilde{\Psi }(x)\mathrm{d}\nu
_{1}(x)-\int_{0}^{1}\Psi _{0}(x)\mathrm{d}\nu _{1}(x).%
\end{array}%
\right.
\end{equation*}%
This yields that
\begin{equation*}
\widetilde{\Psi }(x)=-g_{0}(x)\int_{0}^{1}\Psi _{0}(x)\mathrm{d}\nu
_{0}(x)-g_{1}(x)\int_{0}^{1}\Psi _{0}(x)\mathrm{d}\nu _{1}(x).
\end{equation*}%
Thus for every $\Psi \in $ dom$(L)$ and $\lambda \in
\mathbb{C}
\backslash \left[ \sigma (L_{0})\cup \sigma _{p}(L)\right] ,$ it follows
from $(\ref{resolvent decom})$ that%
\begin{eqnarray*}
&&\left[ R_{\lambda }(L)(L-\lambda )\Psi \right] (x) \\
&=&\Psi _{0}(x)+g_{0}(x)\int_{0}^{1}\left[ R_{\lambda }(L_{0})(L-\lambda
)\Psi \right] (x)\mathrm{d}\nu _{0}(x) \\
&&+g_{1}(x)\int_{0}^{1}\left[ R_{\lambda }(L_{0})(L-\lambda )\Psi \right] (x)%
\mathrm{d}\nu _{1}(x) \\
&=&\Psi _{0}(x)-\widetilde{\Psi }(x)=\Psi (x).
\end{eqnarray*}%
This completes the proof.
\end{proof}

Now we are in a position to prove Proposition \ref{bounded}.

\begin{proof}[Proof of Proposition \protect\ref{bounded}]
From Proposition \ref{resolvent}, it follows that the resolvent $R_{\lambda
}( L) $ of $L$ is a compact operator, thus the operator $L$ is closed and
its spectrum is purely discrete. Moreover, for any point $\lambda _{0}\in
\sigma ( L) ,$ from Lemma \ref{resolvent} we know that $\lambda _{0}$ is a
pole of $R_{\lambda }( L) $ of finite order. Thus the finiteness of $\alpha
(L-\lambda _{0}I)$ follows from \cite[Chap. V, Theorem 10.1]{Taylor}. This
proves Proposition \ref{bounded}.
\end{proof}

In addition, let us recall several facts which will be used in the next
section.

\begin{lemma}
\label{L2} The initial problem consisting of equation $(\ref{fangcheng})$
and the initial conditions
\begin{equation}
y(0,\lambda )=h,\text{ }y^{\prime }(0,\lambda )=k,  \label{6}
\end{equation}%
where $h,$ $k\in\mathbb{C} $, has a unique solution $y(x,\lambda )$. And
each of the functions $y(x,\lambda )$ and $y^{\prime }(x,\lambda )$ is
continuous on $[0,1]\times
\mathbb{C}
,$ in particular, the functions $y(x,\lambda )$ and $y^{\prime }(x,\lambda )$
are entire functions of $\lambda \in
\mathbb{C}
.$
\end{lemma}

\begin{proof}
See \cite{CC7}.
\end{proof}

\begin{remark}
\label{derivative} In fact, from \cite{CC7} one also has the derivative of $%
y(x,\lambda )$ with respect to $\lambda $ is given by
\begin{equation*}
y_{\lambda }^{\prime }(x,\lambda )=\int_{0}^{x}\frac{y_{2}(x,\lambda
)y_{1}(t,\lambda )-y_{1}(x,\lambda )y_{2}(t,\lambda )}{b_{0}(t)\exp (
-\int_{0}^{t}\frac{b_{1}(s)}{b_{0}(s)}\mathrm{d}s) }y(t,\lambda )\mathrm{d}t.
\end{equation*}
\end{remark}

\begin{remark}
Lemma $\ref{L2}$ implies that $\Delta (\lambda )$ is an entire function of $%
\lambda \in
\mathbb{C}
.$
\end{remark}

\begin{remark}
\label{resolvent0}Consider the differential operator $\widetilde{L}_{0}$ in $%
L_{w}^{2}(J,%
\mathbb{C}
)$ defined by
\begin{eqnarray*}
\widetilde{L}_{0}y &:&=b_{0}(x)y^{\prime \prime }+b_{1}(x)y^{\prime }, \\
\text{dom}(\widetilde{L}_{0}) &:&=\left\{ y\in L_{w}^{2}(J,%
\mathbb{C}
)\left\vert
\begin{array}{l}
y,y^{\prime }\in AC[0,1],\widetilde{L}_{0}y\in L_{w}^{2}(J,%
\mathbb{C}
), \\
y(0)=y^{\prime }(0)=0%
\end{array}%
\right. \right\} .
\end{eqnarray*}%
It is obvious that the resolvent set $\rho (\widetilde{L}_{0})=%
\mathbb{C}
.$ Moreover, direct calculation yields that for each $f\in L_{w}^{2}(J,%
\mathbb{C}
)$ and $x\in \left[ 0,1\right] ,$
\begin{equation*}
(R_{\lambda }(\widetilde{L}_{0})f)(x)=\int_{0}^{x}\frac{y_{2}(x,\lambda
)y_{1}(t,\lambda )-y_{1}(x,\lambda )y_{2}(t,\lambda )}{b_{0}(t)\exp
(-\int_{0}^{t}\frac{b_{1}(s)}{b_{0}(s)}\mathrm{d}s)}f(t)\mathrm{d}t.
\end{equation*}%
Therefore, $(R_{\lambda }(\widetilde{L}_{0})f)(x)$ is an entire function of $%
\lambda \in\mathbb{C}.$
\end{remark}

\section{Proof of Theorem \protect\ref{equal} and Remarks}

Based on the statements given in the previous section, we present the proof
of Theorem \ref{equal} in this section and use this result to solve several
problems.

\begin{proof}[Proof of Theorem \protect\ref{equal}]
Let $m_{0}$ denote the ascent of the operator $L-\lambda _{0}I$, then $\chi
(\lambda _{0})=\dim \ker (( L-\lambda _{0}I) ^{m_{0}}).$ Denote the
geometric multiplicity of the eigenvalue $\lambda _{0}$ by $m.$ Then it is
obvious that $m\leq 2.$ Note that we will mainly prove the statement of this
theorem in the case of $m=1,$ since the proof for $m=2$ can be given only
with a slight modification.

When $m=1,$ the proof can be divided into two steps.

\textbf{Step 1. }When $\lambda $ is sufficiently close to $\lambda _{0},$ we
first construct two linear independent solutions $\phi _{1}( x,\lambda ) $
and $\phi _{2}( x,\lambda ) $ of the equation $( l-\lambda I) y=0$ via the
generalized eigenfunctions of $\lambda _{0}.$ Let us recall that $%
ly=b_{0}(x)y^{\prime \prime }+b_{1}(x)y^{\prime }.$

Define a linear operator $F$ on the finite dimensional space $\ker ((
L-\lambda _{0}I) ^{m_{0}})$ as follows:%
\begin{equation*}
F=\left. ( L-\lambda _{0}I) \right\vert \ker (( L-\lambda _{0}I) ^{m_{0}}).
\end{equation*}%
Then $F^{m_{0}}=0$ and $F^{m_{0}-1}\neq 0,$ i.e., $F$ is nilpotent with
index $m_{0}.$ It follows from \cite[Chapter 57, Theorem 2]{vector} that
there exist functions
\begin{equation*}
\eta ,F\eta ,\ldots F^{m_{0}-1}\eta
\end{equation*}%
form a basis of the generalized space $\ker (( L-\lambda _{0}I) ^{m_{0}}).$
Note that in this case $m_{0}=\chi (\lambda _{0}).$ Then denote
\begin{equation*}
\xi _{0,1}:=F^{m_{0}-1}\eta ,\text{ }\xi _{1,1}:=F^{m_{0}-2}\eta ,\cdots
\text{ }\xi _{m_{0}-1,1}:=\eta.
\end{equation*}%
Select another solution $\xi _{0,2}$ of the equation $( l-\lambda _{0}I) y=0$
such that $\xi _{0,1}$ and $\xi _{0,2}$ are fundamental solutions of $(
l-\lambda _{0}I) y=0.$

For $\lambda \in
\mathbb{C}
,$ define
\begin{eqnarray}
\phi _{1}(x,\lambda ) &=&\sum_{k=0}^{m_{0}-1}(\lambda -\lambda _{0})^{k}\xi
_{k,1}(x)+(\lambda -\lambda _{0})^{m_{0}}(R_{\lambda }(\widetilde{L}%
_{0})\eta )(x),  \label{definition solution} \\
\phi _{2}(x,\lambda ) &=&\xi _{0,2}(x)+(\lambda -\lambda _{0})(R_{\lambda }(%
\widetilde{L}_{0})\xi _{0,2})(x).  \label{definition solution1}
\end{eqnarray}%
Note that $\phi _{i}(x,\lambda _{0})=\xi _{0,i}(x),$ $i=1,2.$ Then we will
show that $\phi _{1}(x,\lambda )$ and $\phi _{2}(x,\lambda )$ are linear
independent solutions of the equation $(l-\lambda I)y=0.$ In fact,%
\begin{eqnarray*}
&&((l-\lambda I)\phi _{1})(x,\lambda ) \\
&=&\sum_{k=1}^{m_{0}-1}(\lambda -\lambda _{0})^{k}((l-\lambda _{0}I)\xi
_{k,1})(x)-\sum_{k=0}^{m_{0}-1}(\lambda -\lambda _{0})^{k+1}\xi
_{k,1}(x)+(\lambda -\lambda _{0})^{m_{0}}\eta (x) \\
&=&\sum_{k=1}^{m_{0}-1}(\lambda -\lambda _{0})^{k}\xi
_{k-1,1}(x)-\sum_{k=0}^{m_{0}-1}(\lambda -\lambda _{0})^{k+1}\xi
_{k,1}(x)+(\lambda -\lambda _{0})^{m_{0}}\xi _{m_{0}-1,1}(x) \\
&=&0
\end{eqnarray*}%
and%
\begin{equation*}
((l-\lambda I)\phi _{2})(x,\lambda )=((l-\lambda I)\xi _{0,2})(x)+(\lambda
-\lambda _{0})\xi _{0,2}(x)=0.
\end{equation*}%
Moreover, since $\xi _{0,1}$ and $\xi _{0,2}$ are linear independent
solutions of $(l-\lambda _{0}I)y=0,$ we have
\begin{equation*}
\det \left(
\begin{array}{cc}
\phi _{1}(0,\lambda _{0}) & \phi _{2}(0,\lambda _{0}) \\
\phi _{1}^{\prime }(0,\lambda _{0}) & \phi _{2}^{\prime }(0,\lambda _{0})%
\end{array}%
\right) =\det \left(
\begin{array}{cc}
\xi _{0,1}(0) & \xi _{0,2}(0) \\
\xi _{0,1}^{\prime }(0) & \xi _{0,2}^{\prime }(0)%
\end{array}%
\right) \neq 0.
\end{equation*}%
As a consequence of Remark \ref{resolvent0}, $\phi _{i}(0,\lambda )$ and $%
\phi _{i}^{\prime }(0,\lambda ),$ $i=1,2$ are entire functions of $\lambda
\in
\mathbb{C}
,$ hence there exists a number $\delta >0$ such that for $\left\vert \lambda
-\lambda _{0}\right\vert <\delta ,$ $\phi _{1}(x,\lambda )$ and $\phi
_{2}(x,\lambda )$ are linear independent solutions of $(l-\lambda I)y=0.$

\textbf{Step 2. }Based on step 1, when $\left\vert \lambda -\lambda
_{0}\right\vert <\delta ,$ one has
\begin{equation}
\left(
\begin{array}{c}
y_{1}(x,\lambda ) \\
y_{2}(x,\lambda )%
\end{array}%
\right) =\left(
\begin{array}{cc}
b_{11}(\lambda ) & b_{12}(\lambda ) \\
b_{21}(\lambda ) & b_{22}(\lambda )%
\end{array}%
\right) \left(
\begin{array}{c}
\phi _{1}(x,\lambda ) \\
\phi _{2}(x,\lambda )%
\end{array}%
\right)  \label{linear}
\end{equation}%
and $\det \left(
\begin{array}{cc}
b_{11}(\lambda ) & b_{12}(\lambda ) \\
b_{21}(\lambda ) & b_{22}(\lambda )%
\end{array}%
\right) \neq 0.$ Thus from $(\ref{linear}),$ the definition of $\phi _{1},$ $%
\phi _{2},$ and the fact $\xi _{k,1}\in $ dom$(L-\lambda _{0}I),\
k=0,1,\ldots m_{0}-1,$ it follows that
\begin{eqnarray*}
&&\Delta (\lambda ) \\
&=&\det \left(
\begin{array}{cc}
\int_{0}^{1}y_{1}(x,\lambda )\mathrm{d}\nu _{0}(x)-y_{1}(0,\lambda ) &
\int_{0}^{1}y_{2}(x,\lambda )\mathrm{d}\nu _{0}(x)-y_{2}(0,\lambda ) \\
\int_{0}^{1}y_{1}(x,\lambda )\mathrm{d}\nu _{1}(x)-y_{1}(1,\lambda ) &
\int_{0}^{1}y_{2}(x,\lambda )\mathrm{d}\nu _{1}(x)-y_{2}(1,\lambda )%
\end{array}%
\right) \\
&=&(\lambda -\lambda _{0})^{m_{0}}\det (g_{i,j}(\lambda ))\det \left(
\begin{array}{cc}
b_{11}(\lambda ) & b_{21}(\lambda ) \\
b_{12}(\lambda ) & b_{22}(\lambda )%
\end{array}%
\right) ,
\end{eqnarray*}%
where $\left\vert \lambda -\lambda _{0}\right\vert <\delta ,$ and for $%
i=1,2, $
\begin{eqnarray*}
g_{i,1}(\lambda ) &=&\int_{0}^{1}(R_{\lambda }(\widetilde{L}_{0})\eta )(x)%
\mathrm{d}\nu _{i-1}(x)-(R_{\lambda }(\widetilde{L}_{0})\eta )(i-1), \\
g_{i,2}(\lambda ) &=&(\lambda -\lambda _{0})\left[ \int_{0}^{1}(R_{\lambda }(%
\widetilde{L}_{0})\xi _{0,2})(x)\mathrm{d}\nu _{i-1}(x)-(R_{\lambda }(%
\widetilde{L}_{0})\xi _{0,2})(i-1)\right] \\
&&+\int_{0}^{1}\xi _{0,2}(x)\mathrm{d}\nu _{i-1}(x)-\xi _{0,2}(i-1).
\end{eqnarray*}%
Recall that in this case $m_{0}=\chi (\lambda _{0}).$ Thus in order to show
the order of $\lambda _{0}$ as a zero of $\Delta (\lambda )$ is equal to $%
\chi (\lambda _{0}),$ it is sufficient to prove that $\det (g_{i,j}(\lambda
_{0}))\neq 0$ since $g_{i,j}(\lambda )$ are entire functions of $\lambda \in
\mathbb{C}
.$ Otherwise, there exists a constant $c$ such that
\begin{eqnarray*}
&&\left(
\begin{array}{c}
\int_{0}^{1}(R_{\lambda _{0}}(\widetilde{L}_{0})\eta )(x)\mathrm{d}\nu
_{0}(x)-(R_{\lambda _{0}}(\widetilde{L}_{0})\eta )(0) \\
\int_{0}^{1}(R_{\lambda _{0}}(\widetilde{L}_{0})\eta )(x)\mathrm{d}\nu
_{1}(x)-(R_{\lambda _{0}}(\widetilde{L}_{0})\eta )(1)%
\end{array}%
\right) \\
&=&c\left(
\begin{array}{c}
\int_{0}^{1}\xi _{0,2}(x)\mathrm{d}\nu _{0}(x)-\xi _{0,2}(0) \\
\int_{0}^{1}\xi _{0,2}(x)\mathrm{d}\nu _{1}(x)-\xi _{0,2}(1)%
\end{array}%
\right) .
\end{eqnarray*}%
Denote $u(x)=(R_{\lambda _{0}}(\widetilde{L}_{0})\eta )(x)-c\xi _{0,2}(x),$
then the above equation implies that
\begin{equation*}
\int_{0}^{1}u(x)\mathrm{d}\nu _{0}(x)=u(0)\text{ and}\int_{0}^{1}u(x)\mathrm{%
d}\nu _{1}(x)=u(1).
\end{equation*}%
Therefore,
\begin{equation}
(l-\lambda _{0}I)u=\eta \in \ker ((L-\lambda _{0}I)^{m_{0}})  \label{uuu}
\end{equation}%
and hence
\begin{equation*}
u\in \ker ((L-\lambda _{0}I)^{m_{0}+1})=\ker ((L-\lambda _{0}I)^{m_{0}}).
\end{equation*}%
This implies that there exists constants $\alpha _{i}$ such that $%
u=\sum\limits_{i=0}^{m_{0}-1}\alpha _{i}F^{i}\eta .\ $Thus
\begin{equation*}
(l-\lambda _{0}I)u=\sum\limits_{i=0}^{m_{0}-1}\alpha _{i}(l-\lambda
_{0}I)F^{i}\eta =\sum\limits_{i=0}^{m_{0}-2}\alpha _{i}F^{i+1}\eta
=\sum\limits_{i=1}^{m_{0}-1}\alpha _{i-1}F^{i}\eta .
\end{equation*}%
This together with $(\ref{uuu})$ yield that $\eta
=\sum\limits_{i=1}^{m_{0}-1}\alpha _{i-1}F^{i}\eta $ which contradicts the
linear independence of $\eta ,F\eta ,\ldots F^{m_{0}-1}\eta $. This proves $%
\det (g_{i,j}(\lambda _{0}))\neq 0.$ Hence the statement of Theorem \ref%
{equal} in the case of $m=1$ is proved.

Now we turn to the case $m=2.$ We only need to make slight modifications on
the solutions $\phi _{1},$ $\phi _{2}\ $and $(g_{i,j}(\lambda )).$ Note that
it follows from \cite[Chapter 57, Theorem 2]{vector} that there exists
functions $\eta _{1,}$ $\eta _{2}\in \ker ((L-\lambda _{0}I)^{m_{0}})$ such
that
\begin{eqnarray*}
&&\eta _{1},F\eta _{1},\ldots \ F^{q_{1}-1}\eta _{1}, \\
&&\eta _{2},F\eta _{2},\ldots \ F^{q_{2}-1}\eta _{2}
\end{eqnarray*}%
form a basis of the generalized space $\ker ((L-\lambda _{0}I)^{m_{0}})$
where $q_{1}+q_{2}=\chi (\lambda _{0})$, $m_{0}=q_{1}\geq q_{2}>0$ and $%
F^{q_{1}}\eta _{1}=F^{q_{2}}\eta _{2}=0.$ In this case, denote
\begin{equation*}
\xi _{0,i}:=F^{q_{i}-1}\eta _{i},\text{ }\xi _{1,i}:=F^{q_{i}-2}\eta
_{i},\cdots \text{ }\xi _{m_{0}-1,i}:=\eta _{i},\ i=1,2.
\end{equation*}%
Hence $\xi _{0,1}$ and $\xi _{0,2}$ are fundamental solutions of $(l-\lambda
_{0}I)y=0.$

For $\lambda \in
\mathbb{C}
,$ define
\begin{equation}
\phi _{i}(x,\lambda )=\sum_{k=0}^{q_{i}-1}(\lambda -\lambda _{0})^{k}\xi
_{k,i}(x)+(\lambda -\lambda _{0})^{m_{0}}(R_{\lambda }(\widetilde{L}%
_{0})\eta _{i})(x),\ i=1,2.
\end{equation}%
Note that $\phi _{i}(x,\lambda _{0})=\xi _{0,i}(x),$ $i=1,2.$ By a process
similar to that in the case $m=1$, it can be obtained that
\begin{equation*}
g_{i,j}(\lambda )=\int_{0}^{1}(R_{\lambda }(\widetilde{L}_{0})\eta _{j})(x)%
\mathrm{d}\nu _{i-1}(x)-(R_{\lambda }(\widetilde{L}_{0})\eta _{j})(i-1).
\end{equation*}%
Similarly, $\det (g_{i,j}(\lambda _{0}))\neq 0$. Otherwise, there exists a
constant $c$ such that
\begin{eqnarray*}
&&\left(
\begin{array}{c}
\int_{0}^{1}(R_{\lambda _{0}}(\widetilde{L}_{0})\eta _{1})(x)\mathrm{d}\nu
_{0}(x)-(R_{\lambda _{0}}(\widetilde{L}_{0})\eta _{1})(0) \\
\int_{0}^{1}(R_{\lambda _{0}}(\widetilde{L}_{0})\eta _{1})(x)\mathrm{d}\nu
_{1}(x)-(R_{\lambda _{0}}(\widetilde{L}_{0})\eta _{1})(1)%
\end{array}%
\right) \\
&=&c\left(
\begin{array}{c}
\int_{0}^{1}(R_{\lambda _{0}}(\widetilde{L}_{0})\eta _{2})(x)\mathrm{d}\nu
_{0}(x)-(R_{\lambda _{0}}(\widetilde{L}_{0})\eta _{2})(0) \\
\int_{0}^{1}(R_{\lambda _{0}}(\widetilde{L}_{0})\eta _{2})(x)\mathrm{d}\nu
_{1}(x)-(R_{\lambda _{0}}(\widetilde{L}_{0})\eta _{2})(1)%
\end{array}%
\right) .
\end{eqnarray*}%
Denote $u(x)=(R_{\lambda _{0}}(\widetilde{L}_{0})\eta _{1})(x)-c(R_{\lambda
_{0}}(\widetilde{L}_{0})\eta _{2})(x),$ then $\int_{0}^{1}u(x)\mathrm{d}\nu
_{0}(x)=u(0)$ and $\int_{0}^{1}u(x)\mathrm{d}\nu _{1}(x)=u(1).$ Hence
\begin{equation*}
(l-\lambda _{0}I)u=\eta _{1}-c\eta _{2}\in \ker ((L-\lambda _{0}I)^{m_{0}})
\end{equation*}%
and
\begin{equation*}
u\in \ker ((L-\lambda _{0}I)^{m_{0}+1})=\ker ((L-\lambda _{0}I)^{m_{0}}).
\end{equation*}%
This implies that there exist constants $\alpha _{i},$ $\beta _{i}$ such
that $u=\sum\limits_{i=0}^{q_{1}-1}\alpha _{i}F^{i}\eta
_{1}+\sum\limits_{i=0}^{q_{2}-1}\beta _{i}F^{i}\eta _{2}.$ Hence
\begin{equation*}
\eta _{1}-c\eta _{2}=(l-\lambda _{0}I)u=\sum\limits_{i=1}^{q_{1}-1}\alpha
_{i-1}F^{i}\eta _{1}+\sum\limits_{i=1}^{q_{2}-1}\beta _{i-1}F^{i}\eta _{2}.
\end{equation*}%
This contradicts the linear independence of $\eta _{i},F\eta _{i},\ldots \
F^{q_{i}-1}\eta _{i},i=1,2$. Now the proof is completed.
\end{proof}

Based on Theorem \ref{equal}, we conclude this paper with three remarks on
two concrete eigenvalue problems which have been treated in \cite%
{diffusion,kolb3,drift}.

\begin{remark}
\label{extend diffusion}Consider the eigenvalue problem with coefficients $%
b_{0}\equiv -1,$ $b_{1}\equiv 0$ and $\nu _{0}=\nu _{1}=\delta _{a},$ $a\in
(0,1),$ i.e.,
\begin{equation}
\left\{
\begin{array}{l}
-y^{\prime \prime }(x)=\lambda y(x),\text{ }x\in (0,1), \\
y(0)=y(a)=y(1),\text{ }a\in (0,1).%
\end{array}%
\right.  \label{dexin}
\end{equation}%
In \cite[Theorem 1]{diffusion}, M. Kolb and D. Krej\v{c}i\v{r}\'{\i}k showed
all the eigenvalues of the problem $(\ref{dexin})$ are algebraically simple
if, and only if, $a\notin $ $%
\mathbb{Q}
.$ Based on Theorem \ref{equal}, this interesting result can be obtained
from a different perspective.
\end{remark}

In fact, from $(\ref{chara})$ we know that $\lambda $ is an eigenvalue of
the problem $(\ref{dexin})$ if and only if
\begin{equation*}
\Delta (\lambda )=-\frac{4}{\sqrt{\lambda }}\sin \frac{\sqrt{\lambda }(1-a)}{%
2}\sin \frac{\sqrt{\lambda }a}{2}\sin \frac{\sqrt{\lambda }}{2}=0.
\end{equation*}%
Furthermore,
\begin{eqnarray*}
\Delta ^{\prime }(\lambda )&=&\frac{2}{\lambda ^{\frac{3}{2}}}\sin \frac{%
\sqrt{\lambda }(1-a)}{2}\sin \frac{\sqrt{\lambda }a}{2}\sin \frac{\sqrt{%
\lambda }}{2}-\frac{1-a}{\lambda }\cos \frac{\sqrt{\lambda }(1-a)}{2}\sin
\frac{\sqrt{\lambda }a}{2}\sin \frac{\sqrt{\lambda }}{2} \\
&&-\frac{a}{\lambda }\sin \frac{\sqrt{\lambda }(1-a)}{2}\cos \frac{\sqrt{%
\lambda }a}{2}\sin \frac{\sqrt{\lambda }}{2}-\frac{1}{\lambda }\sin \frac{%
\sqrt{\lambda }(1-a)}{2}\sin \frac{\sqrt{\lambda }a}{2}\cos \frac{\sqrt{%
\lambda }}{2}.
\end{eqnarray*}%
Thus it is easy to see that for any $\widehat{\lambda }\neq 0,$ $\Delta (%
\widehat{\lambda })=\Delta ^{\prime }(\widehat{\lambda })=0$ if and only if
\begin{equation*}
\sin \frac{\sqrt{\widehat{\lambda }}(1-a)}{2}=\sin \frac{\sqrt{\widehat{%
\lambda }}a}{2}=\sin \frac{\sqrt{\widehat{\lambda }}}{2}=0,
\end{equation*}%
i.e.,
\begin{equation}
\widehat{\lambda }=(2m\pi )^{2}=(\frac{2n\pi }{a})^{2}=(\frac{2l\pi }{1-a}%
)^{2},\text{ }m,\text{ }n,\text{ }l\in
\mathbb{N}
:=\left\{ 1,2,\ldots \right\} .  \label{irra}
\end{equation}%
Hence for each $\widehat{\lambda }\neq 0\ $which satisfies $\Delta (\widehat{%
\lambda })=\Delta ^{\prime }(\widehat{\lambda })=0,$ direct calculation
yields that $\Delta ^{\prime \prime }(\widehat{\lambda })=0\ $and $\Delta
^{\prime \prime \prime }(\widehat{\lambda })=\frac{3a^{2}-3a}{8\widehat{%
\lambda }^{2}}\neq 0.$ Obviously, $\Delta ^{\prime }(0)=\frac{a\left(
a-1\right) }{2}\neq 0.$ Thus it follows from Theorem \ref{equal} that the
algebraic multiplicity of each eigenvalue of the problem $(\ref{dexin})$ is
either one or three. Moreover, since $(\ref{irra})$ implies that $a=\frac{n}{%
m}=1-\frac{l}{m}\in
\mathbb{Q}
,$ we can easily conclude that all the eigenvalues of the problem $(\ref%
{dexin})$ are algebraically simple if, and only if, $a\notin $ $%
\mathbb{Q}
.$

\begin{remark}
\label{open diffusion}Consider the eigenvalue problem with constant
coefficients $b_{0}<0,\ b_{1}\in \mathbb{%
\mathbb{R}
}$ and $\nu _{0}=\nu _{1}=\delta _{\frac{1}{2}},$ i.e.,%
\begin{equation}
\left\{
\begin{array}{l}
b_{0}y^{\prime \prime }(x)+b_{1}y^{\prime }(x)=\lambda y(x),\text{ }x\in
(0,1), \\
y(0)=y(\frac{1}{2})=y(1).%
\end{array}%
\right.  \label{de}
\end{equation}%
Assume $b_{1}\neq 0.$ It follows from Theorem $\ref{equal}$ that each
eigenvalue of the problem $(\ref{de})$ is algebraically and geometrically
simple. This partially answers an open question posed by M. Kolb and D. Krej%
\v{c}i\v{r}\'{\i}k \cite[Section 8]{diffusion}.
\end{remark}

In fact, under the transformation $v(x)=\exp (\frac{b_{1}x}{2b_{0}})y(x),$
problem $(\ref{de})$ is equivalent to the following eigenvalue problem%
\begin{equation}
\left\{
\begin{array}{l}
-v^{\prime \prime }(x)+qv(x)=-\frac{\lambda }{b_{0}}v(x),\text{ }x\in (0,1),
\\
v(0)=Av(\frac{1}{2})=\text{ }A^{2}v(1)%
\end{array}%
\right.  \label{vde}
\end{equation}%
where $q=\frac{1}{4}(\frac{b_{1}}{b_{0}})^{2}\ $and $A=\exp (\frac{b_{1}}{%
-4b_{0}}).$ Let $v_{1}(x,\lambda )$ and $v_{2}(x,\lambda )$ be the
fundamental solutions of the differential equation in $(\ref{vde})$ with the
initial conditions%
\begin{equation}
v_{1}(0,\lambda )=v_{2}^{\prime }(0,\lambda )=1,\text{ }v_{2}(0,\lambda
)=v_{1}^{\prime }(0,\lambda )=0,\ \lambda \in
\mathbb{C}
.
\end{equation}%
Then $v_{1}(x,\lambda )=\cos \left( \sqrt{-\frac{\lambda }{b_{0}}-q}x\right)
,$ $v_{2}(x,\lambda )=\frac{\sin \left( \sqrt{-\frac{\lambda }{b_{0}}-q}%
x\right) }{\sqrt{-\frac{\lambda }{b_{0}}-q}}.$ It can be easily obtained
that $\lambda $ is an eigenvalue of the problem $(\ref{de})$ or $(\ref{vde})$
if and only if
\begin{equation*}
\Delta _{1}(\lambda )=\det \left(
\begin{array}{cc}
Av_{1}(\frac{1}{2},\lambda )-1 & Av_{2}(\frac{1}{2},\lambda ) \\
1-A^{2}v_{1}(1,\lambda ) & -A^{2}v_{2}(1,\lambda )%
\end{array}%
\right) =0.
\end{equation*}%
For simplicity, let $u=-\frac{\lambda }{b_{0}}-q,$ then denote
\begin{equation}
\widetilde{\Delta }_{1}(u):=\Delta _{1}(-b_{0}\left( u+q\right) )=-2A^{2}%
\frac{\sin \frac{\sqrt{u}}{2}}{\sqrt{u}}\left( \frac{A^{2}+1}{2A}-\cos \frac{%
\sqrt{u}}{2}\right) .  \label{delta1}
\end{equation}%
When $b_{1}\neq 0,$ it is obvious that $\frac{A^{2}+1}{2A}>1.$ Let $u_{n}$
be the zeros of $\Delta _{1}(u).$ Then direct calculation yields that $%
\left\{ u_{n}\right\} =\left\{ u_{n,1}\right\} \cup \left\{ u_{n,2}\right\}
\cup \left\{ u_{n,3}\right\} ,$
\begin{eqnarray*}
u_{n,1} &=&(2n\pi )^{2},\text{ }u_{n,2}=(4n\pi -2ir)^{2},\text{ }n\in
\mathbb{N}
, \\
\text{ }u_{n,3} &=&(4n\pi +2ir)^{2},\text{ }n\in
\mathbb{N}
_{0}:=\left\{ 0,1,2,\ldots \right\} .
\end{eqnarray*}%
Here $r>0\ $and cosh$r=\frac{A^{2}+1}{2A},$ i.e., $r=\frac{b_{1}}{-4b_{0}}.$
Hence eigenvalues $\lambda _{n}$ of the problem $(\ref{de})$ or $(\ref{vde})$
are as follows:\ $\left\{ \lambda _{n}\right\} =\left\{ \lambda
_{n,1}\right\} \cup \left\{ \lambda _{n,2}\right\} \cup \left\{ \lambda
_{n,3}\right\} ,$
\begin{eqnarray}
\lambda _{n,1} &=&-4b_{0}n^{2}\pi ^{2}-\frac{b_{1}^{2}}{4b_{0}},\text{ }%
\lambda _{n,2}=-16b_{0}n^{2}\pi ^{2}-2b_{1}n\pi i,\text{ }n\in
\mathbb{N}
,  \label{1} \\
\lambda _{n,3} &=&-16b_{0}n^{2}\pi ^{2}+2b_{1}n\pi i,\text{ }n\in
\mathbb{N}
_{0}.  \label{2}
\end{eqnarray}%
For each eigenvalue $\lambda _{n}$ of the problem $(\ref{de}),$ one can
easily obtain $\Delta _{1}^{\prime }(\lambda _{n})\neq 0.$ In order to use
Theorem \ref{equal} to show that each eigenvalue of the problem $(\ref{de})$
is algebraically simple, it is sufficient to show that $\Delta _{1}(\lambda
)\equiv \Delta (\lambda ).$ Note that $\Delta (\lambda )$ is the
characteristic function defined in $(\ref{chara}).$ Denote
\begin{equation*}
\widetilde{y}_{1}(x,\lambda ):=\exp \left( -\frac{b_{1}x}{2b_{0}}\right)
v_{1}(x,\lambda ),\text{ }\widetilde{y}_{2}(x,\lambda ):=\exp \left( -\frac{%
b_{1}x}{2b_{0}}\right) v_{2}(x,\lambda ),
\end{equation*}%
then $\widetilde{y}_{1}(x,\lambda )$ and $\widetilde{y}_{2}(x,\lambda )$ are
solutions of the differential equation in $(\ref{de})$ determined by the
initial conditions%
\begin{equation*}
\widetilde{y}_{1}(0,\lambda )=1,\text{ }\widetilde{y}_{1}^{\prime
}(0,\lambda )=-\frac{b_{1}}{2b_{0}},\text{ }\widetilde{y}_{2}(0,\lambda )=0,%
\text{ }\ \widetilde{y}_{2}^{\prime }(0,\lambda )=1,\lambda \in
\mathbb{C}
.
\end{equation*}%
Thus $\widetilde{y}_{1}(x,\lambda )=$ $y_{1}(x,\lambda )-\frac{b_{1}}{2b_{0}}%
y_{2}(x,\lambda ),$ $\widetilde{y}_{2}(x,\lambda )=$ $y_{2}(x,\lambda ).$
Hence
\begin{eqnarray*}
\Delta _{1}(\lambda ) &=&\det \left(
\begin{array}{cc}
\widetilde{y}_{1}(\frac{1}{2},\lambda )-1 & \widetilde{y}_{2}(\frac{1}{2}%
,\lambda ) \\
1-\widetilde{y}_{1}(1,\lambda ) & -\widetilde{y}_{2}(1,\lambda )%
\end{array}%
\right) \\
&=&\det \left(
\begin{array}{cc}
y_{1}(\frac{1}{2},\lambda )-1 & y_{2}(\frac{1}{2},\lambda ) \\
1-y_{1}(1,\lambda ) & -y_{2}(1,\lambda )%
\end{array}%
\right) =\Delta (\lambda ).
\end{eqnarray*}%
Therefore, each eigenvalue of the problem $(\ref{de})$ is algebraically and
thus geometrically simple.

\begin{remark}
Denote the spectral gap of the problem $( \ref{de}) $ by $\gamma _{1}(
\delta _{\frac{1}{2}}) ,$ i.e.$,$%
\begin{equation*}
\gamma _{1}( \delta _{\frac{1}{2}}) :=\inf \left\{ \left. \text{Re}\lambda
\right\vert \lambda \text{ is an eigenvalue of the problem }( \ref{de})
\text{ and }\lambda \neq 0\text{ }\right\} .
\end{equation*}%
Note that when $b_{1}=0,$ $\lambda _{n}=-4b_{0}n^{2}\pi ^{2},$ $n\in
\mathbb{N}
_{0}$. This together with $( \ref{1}) $ and $( \ref{2}) $ yield
\begin{equation*}
\gamma _{1}( \delta _{\frac{1}{2}}) =\left\{
\begin{array}{c}
-4b_{0}\pi ^{2}-\frac{b_{1}^{2}}{4b_{0}},\text{ when }\left\vert
b_{1}\right\vert \leq -4\sqrt{3}b_{0}\pi , \\
-16b_{0}\pi ^{2},\text{ when }\left\vert b_{1}\right\vert >-4\sqrt{3}%
b_{0}\pi ,%
\end{array}%
\right.
\end{equation*}%
which is already given in \cite{kolb3} and \cite{drift} by different
approaches.
\end{remark}

\end{document}